\documentclass[11pt]{article}
\usepackage[verbose=true,letterpaper]{geometry}
\AtBeginDocument{
  \newgeometry{
    textheight=9in,
    textwidth=6.5in,
    top=1in,
    headheight=14pt,
    headsep=25pt,
    footskip=30pt
  }
}

\usepackage{amsfonts, amsthm}
\usepackage{parskip}
\usepackage{mathtools}
\usepackage{microtype}
\usepackage{algorithm}
\usepackage{algpseudocode}
\usepackage{booktabs}
\usepackage{multirow}
\usepackage{subcaption}
\usepackage{orcidlink}
\usepackage{hyperref}
\usepackage[section]{placeins}
\hypersetup{
  colorlinks=true,
  linkcolor={[rgb]{0,0.2,0.6}},
  citecolor={[rgb]{0,0.6,0.2}},
  filecolor={[rgb]{0.8,0,0.8}},
  urlcolor={[rgb]{0.8,0,0.8}},
  runcolor={[rgb]{0.8,0,0.8}},  
  menucolor={[rgb]{0,0.2,0.6}}, 
  linkbordercolor={[rgb]{0,0.2,0.6}},
  citebordercolor={[rgb]{0,0.6,0.2}},
  filebordercolor={[rgb]{0.8,0,0.8}},
  urlbordercolor={[rgb]{0.8,0,0.8}},
  runbordercolor={[rgb]{0.8,0,0.8}},
  menubordercolor={[rgb]{0,0.2,0.6}}, 
}
\usepackage[nameinlink]{cleveref}
\usepackage{newtxtext,newtxmath}

\def\appendices{\appendix\crefalias{section}{appendix}}

\newtheorem{lemma}{Lemma}

\def\R{\mathbb{R}}

\def\figwidthone{8cm}
\def\figwidthtwo{16.5cm}

\DeclareMathOperator{\spn}{span}
\newcommand{\bx}{\mathbf{x}}
\newcommand{\by}{\mathbf{y}}
\newcommand{\bp}{\mathbf{p}}
\newcommand{\bP}{\mathbf{P}}
\newcommand{\bB}{\mathbf{B}}

\newcommand{\bw}{\mathbf{w}}
\newcommand{\br}{\mathbf{r}}
\newcommand{\bJ}{\mathbf{J}}
\newcommand{\bU}{\mathbf{U}}
\newcommand{\bV}{\mathbf{V}}

\title{Residual subspace evolution strategies for nonlinear inverse problems}
\date{\small (dated: \today)}

\author{\orcidlink{0000-0003-1065-2590}\,Francesco Alemanno\\
\textit{\small SolarSud T.R.E., Via Tiziano Vecellio 11, San Giorgio Ionico (TA), Italy}}

\begin{document}

\maketitle

\textbf{Abstract:} Nonlinear inverse problems pervade engineering and science, yet noisy, non-differentiable, or expensive residual evaluations routinely defeat Jacobian-based solvers. Derivative-free alternatives either demand smoothness, require large populations to stabilise covariance estimates, or stall on flat regions where gradient information fades. 
  
This paper introduces residual subspace evolution strategies (RSES), a derivative-free solver that draws Gaussian probes around the current iterate, records how residuals change along those directions, and recombines the probes through a least-squares solve to produce an optimal update. The method builds a residual-only surrogate without forming Jacobians or empirical covariances, and each iteration costs just $k+1$ residual evaluations with $O(k^3)$ linear algebra overhead, where $k$ remains far smaller than the parameter dimension. 

Benchmarks on calibration, regression, and deconvolution tasks show that RSES reduces misfit consistently across deterministic and stochastic settings, matching or exceeding xNES, NEWUOA, Adam, and ensemble Kalman inversion under matched evaluation budgets. The gains are most pronounced when smoothness or covariance assumptions break, suggesting that lightweight residual-difference surrogates can reliably guide descent where heavier machinery struggles. 


\tableofcontents

\section{Introduction}
\label{sec:introduction}

Many estimation, calibration, and imaging tasks reduce to finding parameter vectors $\bx \in \R^n$ that drive a residual map
\begin{equation}
  \label{eq:inverse_problem}
  F : \R^n \to \R^m, \qquad F(\bx) \approx 0
\end{equation}
toward zero. Classical solvers such as Gauss-Newton, sequential quadratic programming, and trust-region variants perform well when derivatives remain stable and forward models behave smoothly, yet they falter once residuals become noisy, non-differentiable, or expensive to query. Adaptive gradient methods like Adam \cite{kingma2014adam} handle nonconvex objectives with momentum and per-parameter scaling, yet they can stall on broad plateaus or when gradient noise obscures descent directions.

Derivative-free quadratic modelers like NEWUOA \cite{powell2006newuoa} relax the need for derivative access but still assume smoothness, spending many evaluations to maintain reliable local surrogates. Population methods build robustness from sample statistics: ensemble Kalman inversion \cite{evensen2009data} regularises steps with empirical covariances, and natural evolution strategies such as xNES \cite{wierstra2014natural} adapt search directions globally. These approaches usually require sizeable populations to keep their estimates stable, which inflates cost when each residual call is expensive.

Residual subspace techniques point to a lighter alternative. Anderson acceleration \cite{anderson_iterative_1965,walker2011anderson} and direct inversion in the iterative subspace (DIIS) \cite{pulay1982} recycle residual differences as proxies for Jacobian actions inside small affine subspaces, often converging rapidly without explicit derivatives. Although these methods target fixed-point problems rather than general optimization and do not employ ensembles, they demonstrate that residual evaluations alone can recover low-rank Jacobian information.

Building on this insight, the present work introduces residual subspace evolution strategies (RSES), a derivative-free solver designed for noisy inverse problems. Each iteration perturbs the current iterate with $k$ Gaussian draws, stacks the resulting residual differences, and solves a least-squares problem on the probed subspace. The recovered coefficients recombine the probes into an optimal update without forming Jacobians.

This paper derives RSES as the solution of a minimal-residual problem on the probe subspace, provides parameter guidance for the ridge term, probe count, and probe scale, and reports deterministic and Monte Carlo benchmarks on systems ranging from two-parameter calibration to multilayer-perceptron regression tasks with several thousand parameters. Under matched evaluation budgets, RSES reduces residuals more effectively than derivative-free and gradient-based baselines when their model or smoothness assumptions break. The following sections develop the core methodology, present numerical experiments, and conclude with a discussion of limitations and future directions.

\section{Core methodology}
\label{sec:methodology}

The RSES update solves a regularised minimal-residual problem on a span of Gaussian probes. Each iteration samples perturbations around the current iterate, records the corresponding residual changes, and weights the perturbations so that their recombination yields an improved estimate. The notation uses bold symbols for vectors and matrices: $\bx_t$ denotes the current iterate, $\br_t \coloneqq F(\bx_t)$ its residual, $\bP$ collects the probe directions $\bp_i$, $\bB$ stores their residual differences, and $\bw$ holds the recombination weights. A small numerical floor $\epsilon$, set by default to $10^{-8}$, stabilises the ridge update.

\paragraph{Probe surrogate.}

The method frames the inverse problem through the residual energy
\begin{equation}
  \phi(\bx) \coloneqq \frac{1}{2}\lVert F(\bx)\rVert_2^2.
\end{equation}
At iteration $t$, the algorithm draws $k$ Gaussian perturbations with scale $\sigma$, yielding the probe and residual-difference matrices
\begin{eqnarray}
  \bP &=& [\bp_1,\,\ldots,\,\bp_k] \in \R^{n\times k}, \\
  \bB &=& [F(\bx_t+\bp_1)-\br_t,\,\ldots,\,F(\bx_t+\bp_k)-\br_t] \in \R^{m\times k},
\end{eqnarray}
with each probe constructed from a scaled standard normal draw $\mathbf{g}_i\sim\mathcal{N}(0,I_n)$ so that $\bp_i = \sigma \mathbf{g}_i$. A first-order expansion of $F$ around $\bx_t$ with Jacobian $\bJ_t$ produces
\begin{equation}
  \bB \approx \bJ_t \bP,
  \label{eq:B_linearised}
\end{equation}
from which, for any coefficient vector $\bw$, the residual at the displaced point satisfies
\begin{equation}
  F(\bx_t+\bP \bw) \approx \br_t + \bB \bw.
  \label{eq:probe_surrogate}
\end{equation}
This surrogate reuses measured residual differences and keeps the expansion centred at $\bx_t$ without imposing additional constraints.

\paragraph{Ridge recombination.}

The next step minimises $\phi(\bx_t+\bP \bw)$ under the surrogate~\eqref{eq:probe_surrogate} with a small Tikhonov stabiliser:
\begin{equation}
  \min_{\bw\in\R^{k}} \; \frac{1}{2}\lVert \br_t + \bB \bw\rVert_2^2 + \frac{\lambda}{2}\lVert \bw\rVert_2^2.
  \label{eq:ridge_problem}
\end{equation}
The ridge term keeps the Gram matrix well behaved: with $\lambda>0$, the matrix $\bB^\top \bB + \lambda I_{k}$ remains positive definite, and the normal equations reduce to the $k\times k$ system
\begin{equation}
  (\bB^\top \bB + \lambda I_k) \bw = -\bB^\top \br_t,
  \label{eq:ridge_solution}
\end{equation}
which delivers coefficients that balance residual reduction against coefficient magnitude.

\paragraph{State update.}

Applying the recovered weights to the probes yields the parameter update
\begin{equation}
  \Delta \bx = \bP \bw, \qquad \bx_{t+1} = \bx_t + \Delta \bx.
  \label{eq:update}
\end{equation}
The step lives entirely within the probe span, stays centred at $\bx_t$, and mirrors a Gauss-Newton projection on the surrogate rather than on the full parameter space. In the idealised linear regime on the probe span, the ridge solve produces a descent direction that does not raise the residual energy. \Cref{lem:subspace_descent} in \cref{sec:descent_lemma} establishes that the exact surrogate yields $\phi(\bx_{t+1}) \leq \phi(\bx_t)$, with strict descent whenever the projected gradient on the probe span is nonzero.

\subsection{Algorithm}

The derivation above translates directly into a compact iteration that keeps each move inside the span of the current probes. At every step the algorithm evaluates the residual, draws $k$ Gaussian probes, records how the residual changes along them, solves a ridge-stabilised Gram system, and recombines the probes into a step. Algorithm~\ref{alg:kl} summarises this procedure.

\begin{algorithm}
\caption{Residual subspace evolution strategy (RSES)}
\label{alg:kl}
\begin{algorithmic}[1]
\Require Probe count $k\geq1$, ridge scale $\beta>0$, probe scale $\sigma>0$, forward model $F:\R^n\to\R^m$, tolerance $\textrm{tol}>0$, maximum iterations $T$, numerical floor $\epsilon$
\State Initialise $\bx_0$ and set $\br \gets F(\bx_0)$, $\lambda\gets\max(\beta \lVert \br\rVert_2^2,\,\epsilon)$
\For{$t$ from $0$ to $T-1$}
  \If{$\lVert \br\rVert<\textrm{tol}$}
    \State \textbf{break}
  \EndIf
  \State Draw $\mathbf{g}_i\sim\mathcal{N}(0,I_n)$ and set $\bp_i \gets \sigma \mathbf{g}_i$ for $i=1,\ldots,k$
  \State Set $\bP \gets [\bp_1,\,\ldots,\,\bp_k] \in \R^{n\times k}$ and $\bB \gets [F(\bx_t+\bp_1)-\br,\,\ldots,\,F(\bx_t+\bp_k)-\br] \in \R^{m\times k}$
  \State Solve $(\bB^\top \bB + \lambda I_k) \bw = -\bB^\top \br$ for $\bw$
  \State Set $\Delta \bx \gets \bP \bw$,~~ $\bx_{t+1} \gets \bx_t + \Delta \bx$,~~ $\br \gets F(\bx_{t+1})$
\State $\lambda \gets \max(\beta \lVert \br\rVert_2^2,\,\epsilon)$
\EndFor
\State \Return $\bx_T$ \Comment{or alternatively the iterate with the least $\lVert \br\rVert$}
\end{algorithmic}
\end{algorithm}

\subsection{Parameter selection}
\label{sec:param}

The iteration exposes three user choices: the probe count $k$, the probe scale $\sigma$, and the ridge scale $\beta$. A residual-dependent ridge parameter links these through
\begin{equation}
  \label{eq:lambda_schedule}
  \lambda_t = \max(\beta \lVert \br_t\rVert_2^2,\,\epsilon).
\end{equation}

\paragraph{Probe count.} Given an evaluation budget of $N_{\text{eval}}$ residual calls, a practical prescription sets
\begin{equation}
\label{eq:k_prescription}
  k = 4 + \bigl\lfloor 3\log m \bigr\rfloor, \qquad T = \left\lfloor \frac{N_{\text{eval}}}{k+1} \right\rfloor,
\end{equation}
which keeps the linear solve small while widening the span slowly with the residual dimension $m$. This rule lacks deep theoretical justification but proved adequate across the benchmarks reported below. Increasing the coefficient on $\log m$ helps when the residual shows strong anisotropy and the Gram matrix stays well conditioned. The logarithmic scaling loosely aligns with the Johnson-Lindenstrauss lemma when one views the probe subspace as approximating an $m$-dimensional residual space through random projection.

\paragraph{Probe scale.} The probe amplitude should match the problem units. For unit-scaled variables, a value of $\sigma = 0.05$ works well. For unscaled problems, setting $\sigma$ to a modest fraction of the prior standard deviation or the norm of a trusted starting point typically suffices. If the first steps stall, doubling $\sigma$ often helps; if residuals spike, halving it restores stability.

\paragraph{Ridge scale.} The ridge keeps the Gram matrix stable without dominating the residual term. A default of $\beta=10^{-5}$ makes $\lambda_t$ decay with $\lVert \br_t\rVert_2^2$ as the iterate improves. \Cref{sec:ridge_justification} shows that this schedule bounds the coefficient norm $\lVert\bw\rVert_2$ uniformly, independent of the residual matrix $\bB$, while maintaining step sizes proportional to the residual near convergence. Larger $\beta$ damps oscillations at the cost of smaller steps; smaller values accelerate progress when the probes align well but may destabilise the solve. For severely ill-conditioned problems, increasing $\beta$ to $10^{-4}$ or $10^{-3}$ restores robustness.

\Cref{fig:ablation} sweeps these parameters on the nonlinear deconvolution benchmark described in \cref{sec:deconvolution}. Moderate changes leave performance stable; very small probe counts or scales slow progress, while excessively large ridge values yield overly conservative steps.

\begin{figure}
  \centering
  \includegraphics[width=\figwidthtwo]{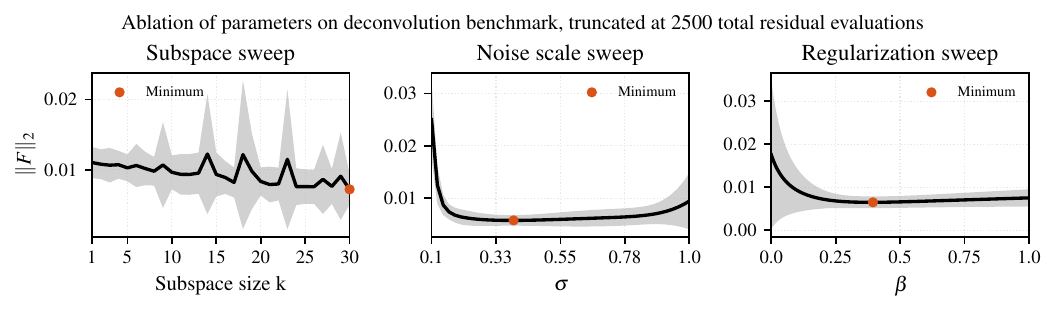}
\caption{Ablation of RSES hyperparameters on the nonlinear deconvolution benchmark. Each panel varies one parameter while the others stay at $k = 30$, $\sigma = 0.9$, and $\beta = 0.2$. Markers highlight the minimum terminal residual per sweep.}
  \label{fig:ablation}
\end{figure}

\subsection{Computational complexity}

Each iteration requires $k+1$ evaluations of $F$, forms $\bB^\top \bB$ and $\bB^\top \br$ in $O(mk^2)$ time, solves a $k\times k$ system in $O(k^3)$ time, and computes $\bP \bw$ in $O(nk)$ time. Storage for $\bP$ and $\bB$ costs $O(nk + mk)$ memory, while the Gram factorization adds $O(k^2)$, yielding total per-iteration memory of $O(nk + mk + k^2)$. The per-iteration computational cost is therefore $O((k+1)C_F + mk^2 + k^3 + nk)$ for residual cost $C_F$. When residual evaluations dominate, the linear algebra overhead remains negligible.

\section{Numerical experiments}
\label{sec:numerics}

All experiments use the parameter choices described in \cref{sec:param}. Each solver receives $7{,}500$ residual evaluations and a maximum runtime of $30$ seconds. The metrics recorded include the best-so-far residual norm, the corresponding parameter error $\lVert \bx - \bx_\star\rVert_2$, and cumulative runtime. RSES sets $k$ from \cref{eq:k_prescription} and maintains a fixed probe scale $\sigma$.

Results average ten independent trials per benchmark. The comparisons pit RSES against NEWUOA \cite{powell2006newuoa}, xNES \cite{wierstra2014natural}, Gauss-Newton (GN) with finite-difference Jacobians, ensemble Kalman inversion (EKI) \cite{iglesias2013ensemble}, and Adam \cite{kingma2014adam} on benchmarks where gradients are available. Tables report mean residual, parameter error, and elapsed time at a shared stopping index: the first evaluation where the mean RSES distance falls within $1\%$ of its minimum. Every solver reports its values at this index, enabling fair comparison at a matched stopping point rather than at arbitrary terminal iterates. Convergence curves show mean best-so-far distance with root-mean-square bands on a shared evaluation axis.

The benchmarks escalate in difficulty, covering: a $2\times 2$ linear system, low-dimensional stochastic Brownian calibration, noisy multilayer-perceptron regression with up to $4{,}353$ parameters, and nonlinear deconvolution with $128$ parameters under both intact and perturbed residual weightings.

\subsection{Linear algebraic system}
\label{sec:linear}

The first test solves the $2\times 2$ linear system
\begin{equation}
  \label{eq:linear_system}
  F(\bx) = \mathbf{A} \bx - \mathbf{R} = 0,
\end{equation}
with
\begin{equation}
  \mathbf{A} = \begin{bmatrix} 101 & -100 \\ 1 & -1 \end{bmatrix}, \qquad \mathbf{R} = \mathbf{A} [1,1]^\top,
\end{equation}
yielding the exact solution $\bx_\star = (1,1)^\top$. RSES and GN reach the solution at numerical precision. RSES attains the smallest residual with runtime on par with GN. NEWUOA approaches the optimum but plateaus with a small residual gap, xNES converges quickly but stalls at higher error, and EKI trails with the largest misfit (\cref{tab:eq_final_err,fig:eq_final_err}).

\begin{table}
  \caption{Terminal residual and state errors for the linear algebraic benchmark.}
  \label{tab:eq_final_err}
  \centering
  \begin{tabular}{l|r|r|r}
\toprule
Algorithm & $\|F(x)\|_2$ & $\|x - x_\star\|_2$ & Time (s) \\
\midrule
RSES & $\mathbf{0}$ & $\mathbf{5.5 \cdot 10^{-15}}$ & $6.7 \cdot 10^{-5}$ \\
NEWUOA & $2.5 \cdot 10^{-3}$ & $3.5 \cdot 10^{-1}$ & $7.8 \cdot 10^{-5}$ \\
xNES & $2.1 \cdot 10^{-1}$ & $1.5 \cdot 10^{0}$ & $\mathbf{5.2 \cdot 10^{-5}}$ \\
GN & $3.1 \cdot 10^{-14}$ & $4.8 \cdot 10^{-14}$ & $7.2 \cdot 10^{-5}$ \\
EKI & $2.8 \cdot 10^{0}$ & $2.4 \cdot 10^{0}$ & $7.6 \cdot 10^{-4}$ \\
\bottomrule
\end{tabular}

\end{table}

\begin{figure}
  \centering
  \includegraphics[width=\figwidthone]{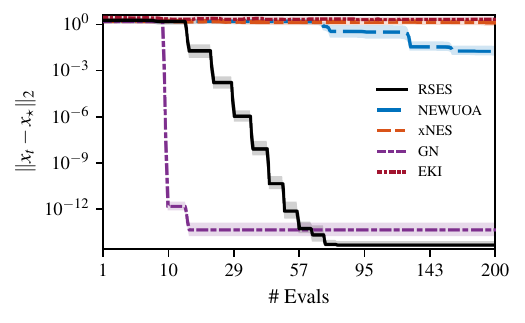}
\caption{Best-iterate convergence on the linear system. RSES and Gauss-Newton recover the solution with nearly identical runtime. RSES attains the smallest residual; NEWUOA improves before plateauing near the optimum; xNES converges quickly but stalls at higher error; and EKI stalls when the probed directions fail to cover the solution.}
  \label{fig:eq_final_err}
\end{figure}

\subsection{Brownian drift and diffusion}
\label{sec:brownian}

The second benchmark estimates the drift and diffusion of a Brownian motion from Monte Carlo summaries. The state follows
\begin{equation}
  dX_t = \mu \, dt + \sigma \, dW_t, \qquad X_0 = 0,
\end{equation}
and the goal is to tune the parameters
\begin{equation}
  \label{eq:brownian_theta}
  \boldsymbol{\theta} \coloneqq (\mu,\log \sigma)
\end{equation}
by matching the empirical terminal statistics
\begin{equation}
  \hat m(\boldsymbol{\theta}) = \tfrac{1}{N} \sum_{j=1}^N X_1^{(j)}, \qquad
  \hat v(\boldsymbol{\theta}) = \tfrac{1}{N} \sum_{j=1}^N \bigl(X_1^{(j)} - \hat m(\boldsymbol{\theta})\bigr)^2
\end{equation}
to reference values $(m_\star,v_\star)$ at
\begin{equation}
  \label{eq:brownian_reference}
  \boldsymbol{\theta}_\star = (0.15,\log 0.35), \qquad F(\boldsymbol{\theta}) = [\hat m(\boldsymbol{\theta}) - m_\star,\, \hat v(\boldsymbol{\theta}) - v_\star].
\end{equation}
Each trajectory evolves via Euler discretisation with $32$ steps of size $1/32$, averaged over $4{,}096$ paths.

RSES tracks the target despite Monte Carlo noise and achieves the smallest residual and parameter error. EKI follows closely in accuracy but runs slower. NEWUOA converges fastest yet stops far from the target. xNES reaches moderate accuracy with cost similar to RSES, while GN overshoots because noisy finite differences distort the Jacobian approximation (\cref{tab:brownian_final_err,fig:brownian_final_err}).

\begin{table}
  \caption{Terminal parameter errors for the Brownian-motion benchmark.}
  \label{tab:brownian_final_err}
  \centering
  \begin{tabular}{l|r|r|r}
\toprule
Algorithm & $\|F(x)\|_2$ & $\|x - x_\star\|_2$ & Time (s) \\
\midrule
RSES & $\mathbf{3.5 \cdot 10^{-4}}$ & $\mathbf{5.3 \cdot 10^{-3}}$ & $1.1 \cdot 10^{-1}$ \\
NEWUOA & $8.9 \cdot 10^{-2}$ & $2.9 \cdot 10^{-1}$ & $\mathbf{2.1 \cdot 10^{-2}}$ \\
xNES & $1.6 \cdot 10^{-2}$ & $4.9 \cdot 10^{-2}$ & $1.1 \cdot 10^{-1}$ \\
GN & $7.0 \cdot 10^{-2}$ & $8.4 \cdot 10^{2}$ & $8.4 \cdot 10^{-2}$ \\
EKI & $5.1 \cdot 10^{-4}$ & $6.9 \cdot 10^{-3}$ & $1.3 \cdot 10^{-1}$ \\
\bottomrule
\end{tabular}

\end{table}

\begin{figure}
  \centering
  \includegraphics[width=\figwidthone]{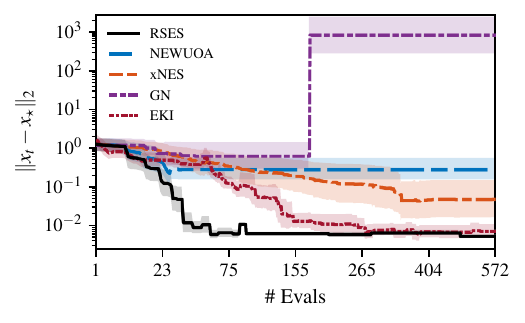}
\caption{Error trajectories for the Brownian calibration. RSES reduces the misfit steadily and ends with the lowest residual and parameter error. EKI tracks it but remains slower, xNES and NEWUOA trail with higher errors, and GN remains unstable.}
  \label{fig:brownian_final_err}
\end{figure}

\subsection{Multilayer perceptron regression}
\label{sec:mlp}

A one-dimensional regression task probes performance on a nonconvex landscape. The target signal follows
\begin{equation}
  \label{eq:mlp_signal}
  y(x) = \sin(3x) + 0.3x
\end{equation}
with additive Gaussian noise of standard deviation $0.05$. A two-hidden-layer multilayer perceptron with widths $8$, $16$, $32$, and $64$ fits these data using the smooth activation
\begin{equation}
  \label{eq:mlp_activation}
  \phi(s) = \frac{s}{\sqrt{1+s^2}},
\end{equation}
which belongs to the softsign family and keeps derivatives bounded \cite{glorot2010understanding}, together with a Tikhonov penalty of $10^{-6}$.

With scalar input $x$, the network computes
\begin{equation}
  \mathbf{h}_1 = \phi(\mathbf{W}_1 x + \mathbf{b}_1), \qquad
  \mathbf{h}_2 = \phi(\mathbf{W}_2 \mathbf{h}_1 + \mathbf{b}_2), \qquad
  \hat y = \mathbf{W}_3 \mathbf{h}_2 + b_3,
\end{equation}
where $\mathbf{W}_1\in\mathbb{R}^{d_1\times 1}$, $\mathbf{W}_2\in\mathbb{R}^{d_2\times d_1}$, $\mathbf{W}_3\in\mathbb{R}^{1\times d_2}$, and $(d_1,d_2)$ takes one of the four width pairs. The Tikhonov residual
\begin{equation}
  F(\boldsymbol{\theta}) =
  \begin{bmatrix}
    \hat y_{\boldsymbol{\theta}}(x) - y \\
    \sqrt{\lambda/2}\,\boldsymbol{\theta}
  \end{bmatrix}
\end{equation}
stacks data misfit and $\ell_2$ regularisation, and the loss reads
\begin{equation}
  \label{eq:mlp_loss}
  \mathcal{L}(\boldsymbol{\theta}) = \lVert F(\boldsymbol{\theta})\rVert_2^2.
\end{equation}
This surface contains broad plateaus and narrow valleys because the activation saturates and the regulariser couples all parameters, so gradient methods can stall even on this small dataset.

RSES uses $k=20$ (following $k = 4 + \lfloor 3\log m \rfloor$ with $m=256$ data points; the Tikhonov terms do not contribute to $m$ since the regularisation adds no problem complexity), $\sigma=0.01$, ridge scale $\beta=10^{-5}$, and approximately $20{,}000$ residual evaluations spread across $952$ iterations. Each iteration evaluates Gaussian probes of the Tikhonov residual without forming gradients. Adam \cite{kingma2014adam} receives the same $20{,}000$ gradient evaluations. The comparison omits xNES, EKI, and GN on this task because their sampling and covariance costs scale poorly with these parameter counts. NEWUOA uses the same evaluation budget but slows sharply, exceeding the $30$-second time limit once the width reaches $16$.

\Cref{tab:mlp_regression} reports final losses $\mathcal{L}(\boldsymbol{\theta})$ and wall-clock time. RSES attains the lowest loss at every width while running faster than Adam and orders of magnitude faster than NEWUOA on wider networks. Adam remains competitive in loss but slower in time. NEWUOA stalls at substantially higher loss once the width reaches $16$.

\Cref{fig:mlp_regression} shows residual histories across widths and the resulting best fits per algorithm. RSES descends steadily while Adam plateaus before recovering. The RSES fits track the noisy target closely without overfitting. These results indicate that RSES navigates the nonconvex MLP landscape reliably at this scale, though larger data regimes will require batching and may still favour tuned gradient-based training.

\begin{table}
  \caption{Final loss $\mathcal{L}(\boldsymbol{\theta})$ and wall-clock time for noisy MLP regression across network widths.}
  \label{tab:mlp_regression}
  \centering
  \begin{tabular}{l|l|l|r|r}
\toprule
Layer sizes & \# Params. & Algorithm & $\mathcal{L}(\theta)$ & Time (s) \\
\midrule
8x8 & $97$ & RSES & 0.5088 & $0.55$ \\
8x8 & $97$ & NEWUOA & 0.8661 & $3.66$ \\
8x8 & $97$ & Adam & 0.5399 & $8.96$ \\
\midrule
16x16 & $321$ & RSES & 0.5065 & $0.32$ \\
16x16 & $321$ & NEWUOA & 1.7980 & $> 30$ \\
16x16 & $321$ & Adam & 0.5390 & $5.58$ \\
\midrule
32x32 & $1153$ & RSES & 0.4833 & $0.79$ \\
32x32 & $1153$ & NEWUOA & 112.9771 & $> 30$ \\
32x32 & $1153$ & Adam & 0.5327 & $7.18$ \\
\midrule
64x64 & $4353$ & RSES & 0.4852 & $1.35$ \\
64x64 & $4353$ & NEWUOA & 155.4996 & $> 30$ \\
64x64 & $4353$ & Adam & 0.5841 & $11.25$ \\
\bottomrule
\end{tabular}

\end{table}

\begin{figure}
  \centering
\includegraphics[width=\figwidthtwo]{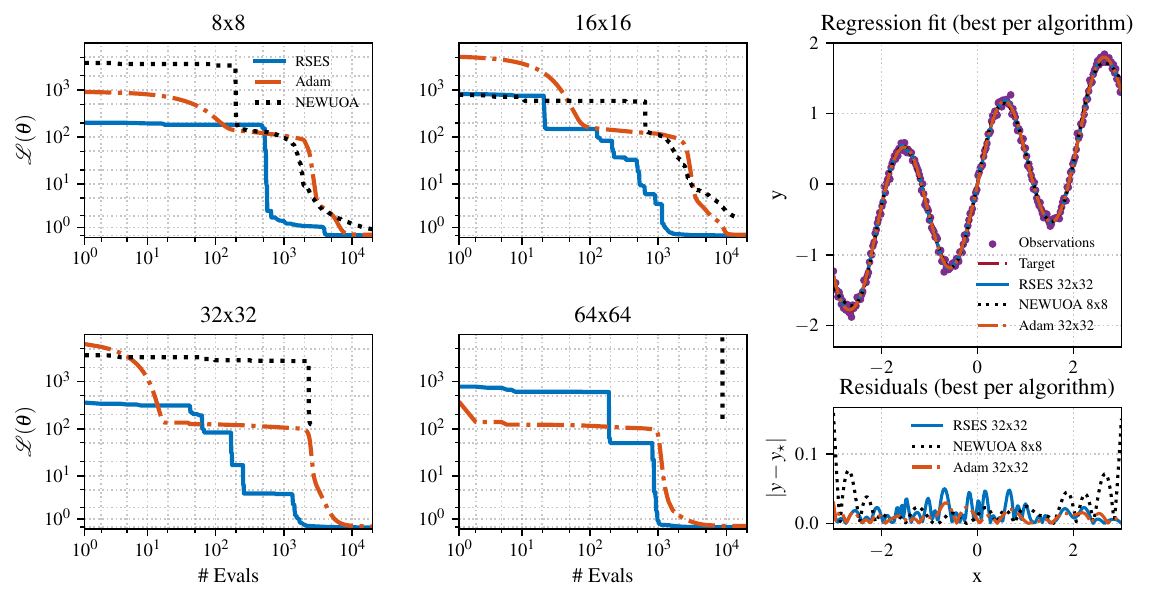}
  \caption{MLP regression under noisy observations. Panels show residual histories for each width and the corresponding best-fit curves per algorithm. RSES decreases the residual steadily while Adam pauses before resuming progress, and the RSES fits align with the noisy target without chasing high-frequency noise.}
  \label{fig:mlp_regression}
\end{figure}

\subsection{Nonlinear deconvolution}
\label{sec:deconvolution}

The final benchmark reconstructs a length-$128$ signal $\bx_\star$ from blurred, saturated, and noisy observations. The forward model applies dense blur, passes the result through a saturating nonlinearity, and adds noise:
\begin{eqnarray}
  g(z) &=& \tanh(z) + \nu(z),\\
  \nu_i &\sim& \mathcal{U}(-5\times 10^{-4},\, 5\times 10^{-4}), \qquad \text{iid per entry and evaluation}.
\end{eqnarray}
A blur matrix $\mathbf{A}\in\mathbb{R}^{n\times n}$ generates the observations
\begin{equation}
  \by = g(\mathbf{A} \bx_\star) + \boldsymbol{\eta}, \qquad \boldsymbol{\eta} \sim \mathcal{N}(0,\sigma^2 I).
\end{equation}
The jitter $\nu$ resamples on every call, introducing a small Monte Carlo perturbation atop the saturating nonlinearity. For each blur the entries of $\mathbf{A}$ are drawn from $\mathcal{N}(0,1/n)$, averaged with the transpose, shifted by $3 I_n$, and scaled by $1/5$ to obtain a well-conditioned low-pass filter. Outputs are rounded to $10^{-3}$ before adding noise, with
\begin{equation}
  \label{eq:sigma_noise}
  \sigma = 0.01 \max_i \lvert g((\mathbf{A} \bx_\star)_i)\rvert.
\end{equation}
This stylised forward model combines blur, saturation, quantisation, and correlated noise typical of imaging pipelines. All solvers evaluate the residual
\begin{equation}
  F(\bx) = \mathbf{B}_{\text{w}} \bigl(g(\mathbf{A} \bx) - \by\bigr),
\end{equation}
with weighting matrix $\mathbf{B}_{\text{w}}\in\mathbb{R}^{3n\times n}$ set to
\begin{equation}
  \label{eq:residual_weightings}
  \mathbf{B}_{\text{w}} = \begin{bmatrix} I \\ 0 \\ 0 \end{bmatrix} \quad \text{(intact residuals)} \quad \text{or} \quad \mathbf{B}_{\text{w}} = \begin{bmatrix} I \\ 0 \\ 0 \end{bmatrix} + \boldsymbol{\Delta}, \qquad \boldsymbol{\Delta}_{ij} \sim \mathcal{N}(0,1/n).
\end{equation}
This construction covers perturbed residuals with correlated components and breaks the covariance pattern across the $3n$ residual entries. The intact case stacks the measurements with two zero blocks; the perturbed case adds Gaussian mixing to every block.

\paragraph{Intact residuals.}
With the intact weighting (\cref{tab:soft_final_err_intact}), RSES achieves near-best state error and near-lowest residual while finishing far faster than every baseline. EKI posts the smallest residual but at much higher computational cost. NEWUOA trails in both accuracy and runtime. xNES remains inaccurate, while GN runs quickly but stays far from the target.

\paragraph{Perturbed residuals.}
With perturbed residuals (\cref{tab:soft_final_err_covariance_perturbed}), RSES leads in residual norm, state error, and runtime among accurate solvers. NEWUOA trails with higher error and cost. xNES and EKI drift from the target, while GN remains fast but inaccurate (\cref{fig:soft_final_err}).

\begin{table}
  \caption{Terminal errors for the nonlinear deconvolution benchmark under intact and perturbed residual weightings.}
  \label{tab:soft_final_err}
  \centering

  \subfloat[a][Intact residual weighting.\label{tab:soft_final_err_intact}]{
    \centering
    \begin{tabular}{l|r|r|r}
\toprule
Algorithm & $\|F(x)\|_2$ & $\|x - x_\star\|_2$ & Time (s) \\
\midrule
RSES & $6.0 \cdot 10^{-3}$ & $2.1 \cdot 10^{-1}$ & $\mathbf{1.7 \cdot 10^{-2}}$ \\
NEWUOA & $1.7 \cdot 10^{0}$ & $4.9 \cdot 10^{0}$ & $6.4 \cdot 10^{-1}$ \\
xNES & $6.8 \cdot 10^{-1}$ & $1.7 \cdot 10^{0}$ & $2.2 \cdot 10^{-1}$ \\
GN & $5.5 \cdot 10^{0}$ & $1.0 \cdot 10^{1}$ & $2.8 \cdot 10^{-2}$ \\
EKI & $\mathbf{3.7 \cdot 10^{-3}}$ & $\mathbf{2.0 \cdot 10^{-1}}$ & $1.0 \cdot 10^{0}$ \\
\bottomrule
\end{tabular}

  }

  \hfill

  \subfloat[b][Perturbed residual weighting.\label{tab:soft_final_err_covariance_perturbed}]{
    \centering
    \begin{tabular}{l|r|r|r}
\toprule
Algorithm & $\|F(x)\|_2$ & $\|x - x_\star\|_2$ & Time (s) \\
\midrule
RSES & $\mathbf{1.1 \cdot 10^{-2}}$ & $\mathbf{2.1 \cdot 10^{-1}}$ & $\mathbf{4.0 \cdot 10^{-2}}$ \\
NEWUOA & $1.9 \cdot 10^{0}$ & $4.5 \cdot 10^{0}$ & $1.5 \cdot 10^{0}$ \\
xNES & $1.3 \cdot 10^{0}$ & $2.0 \cdot 10^{0}$ & $3.4 \cdot 10^{-1}$ \\
GN & $1.1 \cdot 10^{1}$ & $1.0 \cdot 10^{1}$ & $5.8 \cdot 10^{-2}$ \\
EKI & $6.2 \cdot 10^{0}$ & $9.6 \cdot 10^{0}$ & $5.0 \cdot 10^{0}$ \\
\bottomrule
\end{tabular}

  }
\end{table}

\begin{figure}
  \centering
\subcaptionbox{Intact residual weighting.}{\includegraphics[width=\figwidthone]{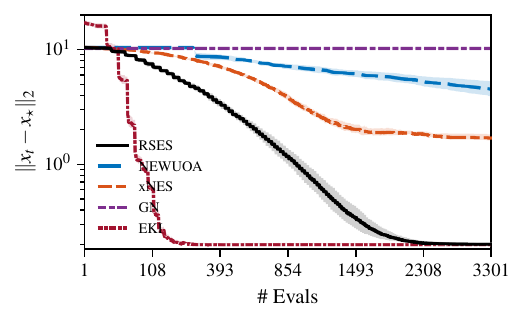}}
\subcaptionbox{Perturbed residual weighting.}{\includegraphics[width=\figwidthone]{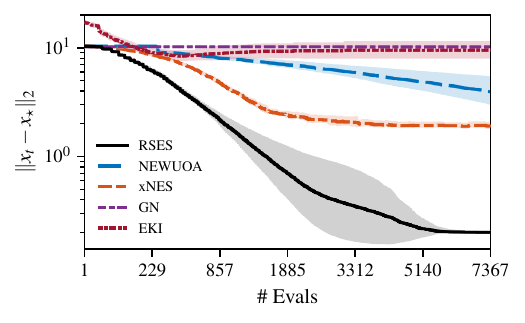}}
\caption{\label{fig:soft_final_err}Best-iterate error trajectories for the deconvolution task. (a) With intact residuals, RSES matches top accuracy while running far faster than every baseline. NEWUOA sits above RSES and EKI, while xNES and GN remain inaccurate. (b) With perturbed residuals, RSES leads in residual and state error with runtime close to GN. NEWUOA lags while xNES and EKI drift from the target.}
\end{figure}

\section{Conclusion}
\label{sec:conclusions}

Residual subspace evolution strategies address nonlinear inverse problems using residual-only information. Across deterministic, noisy, and non-differentiable benchmarks, RSES matched or exceeded every baseline: it solved the linear system to machine precision alongside Gauss-Newton, achieved the best residual and parameter accuracy in Brownian calibration while maintaining reasonable cost, attained the lowest losses on noisy MLP regression while running faster than Adam and far faster than NEWUOA, and ranked at or near the best errors on both deconvolution variants while running substantially faster than EKI.

NEWUOA and related model-based solvers excel on low-dimensional tasks but become costly because they require at least $O(n)$ evaluations per iteration. RSES gains an advantage as dimension increases because it operates with $k \ll n$ probes. Its maintained accuracy under residual weight perturbations underscores robustness to model misspecification.

The method depends on random probes that must explore useful directions. When $k\ll n$, poorly aligned samples can slow progress even with a fixed probe scale. Current parameter selection remains heuristic: the ridge and probe-scale settings do not yet adapt to local curvature or noise levels. Theoretical convergence guarantees under Monte Carlo perturbations remain open.

Future work will address these limitations through adaptive ridge tuning, preconditioned or structured probe sampling, and parallel residual evaluations. Plans include testing RSES on larger imaging and simulation workloads without adjoints, probing objectives with many local minima to assess global robustness, and combining RSES with coarse surrogate models or reduced-order bases when random probes alone cannot efficiently cover the full state space.

\paragraph{Code availability.}
The Julia implementation used for all numerical experiments is freely available at \url{https://doi.org/10.5281/zenodo.17872272}.

\bibliography{bib}

\appendices

\section{Idealised descent on the probe span}
\label{sec:descent_lemma}

\begin{lemma}[Subspace descent]
Let the residual energy be
\begin{equation}
  \label{eq:phi_def_appendix}
  \phi(\bx) = \tfrac{1}{2}\lVert F(\bx)\rVert_2^2.
\end{equation}
Fix an iterate $\bx_t$ and form probes
\begin{equation}
  \label{eq:probe_matrix_appendix}
  \bP = [\bp_1,\ldots,\bp_k], \qquad \bp_i = \sigma \mathbf{g}_i,
\end{equation}
and residual differences
\begin{equation}
  \label{eq:B_differences_appendix}
  \bB = [F(\bx_t+\bp_1)-F(\bx_t),\ldots,F(\bx_t+\bp_k)-F(\bx_t)].
\end{equation}
Assume:
\begin{enumerate}
  \item $F$ is differentiable at $\bx_t$ with Jacobian $\bJ_t$.
  \item The linear relation
  \begin{equation}
    \label{eq:span_linear}
    F(\bx_t+z) = F(\bx_t) + \bJ_t z
  \end{equation}
  holds for every $z\in\spn(\bP)$.
\end{enumerate}
Let $\bw_\star$ minimise
\begin{equation}
  \label{eq:psi_definition}
  \Psi(\bw) = \tfrac{1}{2}\lVert F(\bx_t) + \bB \bw\rVert_2^2 + \tfrac{\lambda}{2}\lVert \bw\rVert_2^2,
\end{equation}
with $\lambda\ge 0$, and set the updates
\begin{equation}
  \label{eq:appendix_update}
  \Delta \bx_\star = \bP \bw_\star, \qquad \bx_{t+1} = \bx_t + \Delta \bx_\star.
\end{equation}
Then $\phi(\bx_{t+1}) \le \phi(\bx_t)$, with strict descent when the projected gradient $G \coloneqq \bP^\top \bJ_t^\top F(\bx_t)$ is nonzero.
\label{lem:subspace_descent}
\end{lemma}

\begin{proof}
Assumption \eqref{eq:span_linear} implies
\begin{equation}
  \label{eq:B_equals_JP}
  \bB = \bJ_t \bP.
\end{equation}
For any $\bw$,
\begin{equation}
  \label{eq:psi_surrogate}
  \Psi(\bw) = \phi(\bx_t + \bP \bw) + \tfrac{\lambda}{2}\lVert \bw\rVert_2^2.
\end{equation}
Evaluating at the origin gives
\begin{equation}
  \label{eq:psi_origin}
  \Psi(0) = \phi(\bx_t).
\end{equation}
Optimality of $\bw_\star$ therefore yields
\begin{equation}
  \label{eq:phi_descent_appendix}
  \phi(\bx_{t+1}) = \Psi(\bw_\star) - \tfrac{\lambda}{2}\lVert \bw_\star\rVert_2^2 \le \Psi(0) = \phi(\bx_t).
\end{equation}
When $G \neq 0$, the gradient of $\Psi$ at the origin satisfies
\begin{equation}
  \label{eq:gradient_identity}
  \nabla \Psi(0) = \bB^\top F(\bx_t) = G \neq 0,
\end{equation}
so the origin cannot minimise $\Psi$. Any minimiser achieves $\Psi(\bw_\star) < \Psi(0)$, which forces $\phi(\bx_{t+1}) < \phi(\bx_t)$.
\end{proof}

\section{Ridge parameter selection}
\label{sec:ridge_justification}

This appendix derives the residual-dependent ridge schedule $\lambda_t = \max(\beta\lVert\br_t\rVert_2^2,\,\epsilon)$ used in the main algorithm.

\paragraph{Uniform bound on the ridge solution.}

The ridge problem~\eqref{eq:ridge_problem} has the closed-form minimiser
\begin{equation}
  \label{eq:w_star_closed}
  \bw^* = -(\bB^\top \bB + \lambda I)^{-1}\bB^\top \br_t.
\end{equation}
Let $\bB = \bU\boldsymbol{\Sigma}\bV^\top$ be a singular value decomposition. Then
\begin{equation}
  (\bB^\top \bB + \lambda I)^{-1}\bB^\top = \bV(\boldsymbol{\Sigma}^2 + \lambda I)^{-1}\boldsymbol{\Sigma}\bU^\top,
\end{equation}
and the operator norm satisfies
\begin{equation}
  \label{eq:operator_norm}
  \lVert(\bB^\top \bB + \lambda I)^{-1}\bB^\top\rVert_2 = \max_{\sigma \ge 0} \frac{\sigma}{\sigma^2 + \lambda}.
\end{equation}
The function $f(\sigma) = \sigma/(\sigma^2 + \lambda)$ attains its maximum at $\sigma = \sqrt{\lambda}$, where $f_{\max} = 1/(2\sqrt{\lambda})$. Hence
\begin{equation}
  \label{eq:w_bound}
  \lVert\bw^*\rVert_2 \le \frac{1}{2\sqrt{\lambda}}\lVert\br_t\rVert_2
\end{equation}
for any matrix $\bB$, and the bound is tight.

\paragraph{Residual-dependent schedule.}

Choosing $\lambda \ge \beta\lVert\br_t\rVert_2^2$ with $\beta > 0$ yields
\begin{equation}
  \lVert\bw^*\rVert_2 \le \frac{1}{2\sqrt{\lambda}}\lVert\br_t\rVert_2 \le \frac{1}{2\sqrt{\beta}},
\end{equation}
so the coefficient norm remains uniformly bounded regardless of the residual magnitude or the structure of $\bB$. To prevent vanishing regularisation when the residual becomes small, the algorithm sets
\begin{equation}
  \label{eq:lambda_final}
  \lambda = \max\{\beta\lVert\br_t\rVert_2^2,\,\epsilon\},
\end{equation}
which ensures
\begin{equation}
  \lVert\bw^*\rVert_2 \le \min\!\left\{\frac{1}{2\sqrt{\beta}},\,\frac{\lVert\br_t\rVert_2}{2\sqrt{\epsilon}}\right\}.
\end{equation}
The first term caps the coefficient norm when the residual is large; the second keeps the step size proportional to the residual once it falls below $\sqrt{\epsilon/\beta}$. This adaptive schedule tightens the ridge as the iterate improves while maintaining numerical stability near convergence.
\end{document}